\documentclass[12pt]{amsart}
\usepackage{amsmath,amsfonts,amscd,amssymb,epsfig,euscript,bm}
\usepackage{amsxtra,mathrsfs}
\usepackage{pdfsync}
\newtheorem{theorem}{Theorem}
\newtheorem{proposition}{Proposition}

\theoremstyle{definition}
{}

\theoremstyle{remark} 
\newtheorem{remark}{Remark}
\newcommand{\field}[1]{\ensuremath{\mathbb{#1}}}
\newcommand{\CC}{\field{C}}

\newcommand{\PP}{\field{P}}

\newcommand{\ZZ}{\field{Z}}
\newcommand{\del}{\partial}

\newcommand{\sd}{\mathrm{d}}

\newcommand{\curly}[1]{\mathscr{#1}}

\newcommand{\cL}{\curly{L}}

\makeatletter
\@namedef{subjclassname@2020}{\textup{2020} Mathematics Subject Classification}
\makeatother
\usepackage{hyperref}
\hypersetup{colorlinks,citecolor=blue,plainpages=false,hypertexnames=false}
\begin{document}
\title[Algebraic de Rham theorem]{Algebraic de Rham theorem and Baker-Akhiezer function}
\author[Igor Krichever]{Igor Krichever$^{\dagger}$}\thanks{$^{\dagger}$October 8, 1950 -- December 1, 2022}
\author{Leon Takhtajan}
\address{Department of Mathematics,
Stony Brook University, Stony Brook, NY 11794 USA; 
\newline
Euler International Mathematical Institute, Pesochnaya Nab. 10, Saint Petersburg 197022 Russia}
\begin{abstract} For the case of algebraic curves --- compact Riemann surfaces --- it is shown that de Rham cohomology group $H^{1}_{\mathrm{dR}}(X,\CC)$ of a genus $g$ Riemann surface $X$ has a natural structure of a symplectic vector space. Every choice of a non-special effective divisor $D$ of degree $g$ on $X$ defines a symplectic basis
of $H^{1}_{\mathrm{dR}}(X,\CC)$, consisting of holomorphic differentials and differentials of the second kind with poles on $D$. This result, the algebraic de Rham theorem, is used to describe the tangent space to Picard and Jacobian varieties of $X$ in terms of differentials of the second kind, and to define a natural vector fields on the Jacobian of $X$ that move 
points of the divisor $D$.  In terms of the Lax formalism on algebraic curves, these vector fields correspond to the Dubrovin equations in the theory of integrable systems, and the Baker-Akhierzer function is naturally obtained by the 
 integration along the integral curves. \end{abstract}
\maketitle
\section{Introduction}
Let $X$ be a smooth algebraic variety over $\CC$ with the classical topology of a complex manifold. According to Atiyah and Hodge \cite{AH}, closed meromorphic $p$-form $\varphi$ on $X$ is called \emph{differential of a second kind}, if it has zero residues on open subsets $U=X\setminus D$ for sufficiently large divisors $D$. A far-reaching generalization of Atiyah and Hodge results was given by Grothendieck \cite{Groth}. The quotient groups 
$$\frac{\{\text{$p$-forms of the second kind}\}}{\{\text{exact forms}\}}$$ 
have a natural interpretation in terms of a spectral sequence of certain complex of sheaves of meromorphic forms on $X$ \cite[Ch. 3, Sect. 5]{GH}. In particular, one gets the statement
$$H_{\mathrm{dR}}^{1}(X,\CC)\simeq \frac{\{\text{$1$-forms of the second kind}\}}{\{\text{exact forms}\}}.$$ 

When $X$ is a smooth algebraic curve of genus $g$, this isomorphism follows from the Riemann-Roch theorem. It turns out that in this case the space of differentials of the second kind carries a
natural skew-symmetric bilinear form, which is non-degenerate when taking a quotient by the subspace of exact forms. Using this bilinear form, in Theorem \ref{Chevalley} we give a more explicit formulation of the algebraic de Rham theorem. Specifically, we show that each non-special effective divisor $D$ of degree $g$ on $X$ defines
a symplectic basis of $H_{\mathrm{dR}}^{1}(X,\CC)$, which makes it possible to explicitly describe a complement to the Lagrangian subspace of holomorphic $1$-forms on $X$ as the subspace of differentials of the second kind with the poles in $D$.

In Section \ref{Picard} it is shown that every non-special effective divisor $D$ of degree $g$ defines an explicit isomorphism between the vector space $H^{0,1}(X,\CC)$ and the Lagrangian subspace of differentials of the second kind with
poles in $D$. This allows us to explicitly describe the tangent space to Picard variety (and its incarnations, Albanese and Jacobian varieties) in pure algebro-geometric terms.

It is quite remarkable that this formalism is connected with the theory of integrable systems. In the standard approach (see, e.g., \cite{FT}) the integrable system is described by the zero curvature equation
$$\frac{\del L}{\del t} -\frac{\del M}{\del x}+LM-ML=0,$$
where $L(x,t,\lambda)$ and $M(x,t,\lambda)$ are certain $r\times r$ matrix-valued rational functions of the spectral parameter $\lambda$, varying on $\CC\PP^{1}$, and also depend on extra arguments $x$ and $t$ (physical space and time variables). In papers \cite{K0,K1} by the first author (I.K.), the zero curvature formalism was
extended  to the case when spectral parameter varies on an algebraic curve. 

Namely, it was shown in \cite{K0} that a natural framework for this generalization is provided by the Hitchin system, expressed in terms of
Tyurin parameters for stable holomorphic vector bundles of rank $r$ and degree $rg$ on an algebraic curve. Correspondingly, the rational functions $L$ on $\CC\PP^{1}$ become special $r\times r$ meromorphic matrix $1$-forms $L(z)dz$
on an algebraic curve, and the set of such matrices is parameterized by the moduli space (or rather its Zariski open subset) of  stable holomorphic vector bundles of rank $r$ and degree $rg$ (see \cite{K0} for details). Similar explicit
description is given for  $r\times r$ meromorphic matrix-valued functions $M(z)$.

It turns out that the simplest case $r=1$ this formalism is still non-trivial and is naturally connected with Theorem \ref{Chevalley} and related discussion in Section \ref{algdR}. Namely, as shown in Section \ref{BA}, the meromorphic 1-forms $L(z)dz$ become differentials of the first kind on an algebraic curve $X$, while analogs of $M(z)$ are meromorphic functions $f$, defined using two non-special effective divisors $D$ and $D_{0}$ of degree $g$ on $X$. Varying divisors $D$ parameterize Jacobian of $X$ with the base point $D_{0}$, and the vector fields describing the motion of points
of $D$ are naturally expressed in terms of the meromorphic functions $f$. 

Remarkably, in the case when $X$ is a hyperelliptic curve, the equations for the integral curves of these vector fields, equations \eqref{dubrovin}, coincide with the \emph{Dubrovin equations},  arising in the theory of finite-gap integration of the Korteweg-de Vries equation \cite{Dubrovin}. Moreover, integrating meromorphic functions $f$ along these integral curves and using Dubrovin's equations, we naturally obtain the  \emph{Baker-Akhiezer function}, a fundamental object in the algebro-geometric approach to integrable systems, introduced by the first author (I.K.) in \cite{K}!
\subsection{Acknowledgments} The second author (L.T.) is grateful to the referee for constructive remarks and suggestions. 
\section{Differentials of the second kind}
Let $X$ be connected, compact Riemann surface of genus $g$ with the classical topology. Denote by $\mathcal{O}_{X}$ the sheaf of germs of holomorphic functions on $X$, by $\mathcal{M}_{X}$ --- the sheaf of germs of meromorphic functions on $X$, and by $\mathcal{M}$ --- the vector space of meromorphic functions on $X$. For every divisor $D$ on $X$ denote by $L=\mathcal{O}(D)$ the holomorphic line bundle associated with $D$, and by $H^{0}(X,L)$ ---  the vector space of holomorphic sections of $L$ over $X$. The following isomorphism
$$H^{0}(X,L)\simeq \mathcal{L}_{D}=\{f\in \mathcal{M} : (f)+D\geq 0\}$$
is very useful.

The Riemann-Roch theorem, together with the Kodaira-Serre duality, is the formula
$$h^{0}(L)-h^{0}(K_{X}-L)=\deg L +1-g,$$
where $h^{0}(L)=\dim_{\CC} H^{0}(X,L)$, $\deg L$ is the degree of $L$, and $K_{X}$ is the canonical class of $X$ --- the holomorphic cotangent bundle to $X$.

Let $\sd$ be the exterior derivative on $X$. The sheaf $\sd \mathcal{M}_{X}$ is a sheaf of germs of differentials of the second kind on $X$ and $\Omega^{(\mathrm{2nd})}=H^{0}(X,\sd \mathcal{M}_{X})$ is the infinite-dimensional vector space of the differentials of the second kind --- meromorphic $1$-forms on $X$ with zero residues.

The infinite-dimensional vector space $\Omega^{(\mathrm{2nd})}$ has a natural skew-symmetric bilinear form\footnote{There are also analogs of the skew-symmetric bilinear form $\omega_{X}$ and of algebraic de Rham theorem for meromorphic quadratic differentials, to appear in a paper by the second author (L.T.).
} 
$$\omega_{X}(\theta_{1},\theta_{2})=\sum_{P\in X}\mathrm{Res}_{P}(\sd^{-1}\theta_{1}\theta_{2}),\quad \theta_{1},\theta_{2}\in \Omega^{(\mathrm{2nd})},$$
where $\sd^{-1}\theta_{1}$ denotes any locally defined function $f$, such that $\sd f = \theta_{1}$, called the local antiderivative. The ambiguity in the choice of $f$ does not matter.

Indeed, it is clear that bilinear form $\omega_{X}$ is defined by a finite sum and the choice of an additive constant in the definition of a local antiderivative is irrelevant. The skew-symmetry of $\omega_{X}$ follows from the basic property
$$\mathrm{Res}_{P}(f_{1}\sd f_{2})=-\mathrm{Res}_{P}(f_{2}\sd f_{1}),$$
where meromorphic functions $f_{1}$ and $f_{2}$ are local antiderivatives of $\theta_{1}$ and $\theta_{2}$ in a neighborhood of $P\in X$.

\section{Algebraic de Rham theorem}\label{algdR} 
In an abstract form, algebraic de Rham theorem is the following statement
\begin{equation}\label{a-dR-1}
H^{1}_{\mathrm{dR}}(X,\CC)\simeq \Omega^{(\mathrm{2nd})}/\sd\mathcal{M},
\end{equation}
which is easily proved using a sheaf-theoretic de Rham isomorphism 
$$H^{1}_{\mathrm{dR}}(X,\CC)\simeq H^{1}(X,\underline{\CC}),$$
where $\underline{\CC}$ is the locally constant sheaf.
 
Indeed, consider the following short exact sequence of sheaves
$$\begin{CD}
0 @>>> \underline{\CC} @>i>> \mathcal{M}_{X} @>\sd>> \sd \mathcal{M}_{X} @>>> 0\\
\end{CD}$$
and the corresponding exact sequence in the cohomology 
$$\begin{CD}
H^{0}(X,\mathcal{M}_{X}) @>\sd>> H^{0}(X,\sd \mathcal{M}_{X}) @>\delta>> H^{1}(X,\underline{\CC})@>>>H^{1}(X, \mathcal{M}_{X}).
\end{CD}$$
It follows from the Riemann-Roch theorem that $H^{1}(X, \mathcal{O}(D))=0$ if $\deg D>2g-2$, which implies (see, e.g., \cite[Ch. 2, \S17.7]{For})
$$H^{1}(X,\mathcal{M}_{X})=\{0\},$$ 
and \eqref{a-dR-1} is proved.

Using bilinear form $\omega_{X}$, we can make isomorphism \eqref{a-dR-1} more concrete. Namely, 
we have the following statement result (see \cite[Ch. 6, \S8]{Ch}, \cite[Ch. III, \S\S5.3-5.4]{Eich} and \cite[Theorem 4]{LT}).
\begin{theorem} \label{Chevalley} The following statements hold.
\begin{enumerate}
\item[(i)] The restriction of the bilinear form $\omega_{X}$ to $\Omega^{(\mathrm{2nd})}/\sd\mathcal{M}$ is non-degenerate and
\begin{displaymath}
\dim_{\CC}\Omega^{(\mathrm{2nd})}/\sd\mathcal{M} =2g.
\end{displaymath}
\item[(ii)]
Each degree $g$ non-special effective divisor $D$ on $X$
defines the isomorphism
\begin{equation*}
\Omega^{(\mathrm{2nd})}/\sd\mathcal{M}\simeq\Omega^{(\mathrm{2nd})}\cap H^0(X, K_{X}+2D).
\end{equation*}
\item[(iii)] Let $D=P_1 + \cdots +P_g$ be a non-special divisor of degree $g$ with distinct points. For every choice of local coordinates in the neighborhoods of $P_i$, the vector space $\Omega^{(\mathrm{2nd})}\cap H^0(X, K_{X}+2D)$ has the
basis $\{\vartheta_i,\tau_i\}_{i=1}^g$, symplectic with respect to the bilinear from $\omega_{X}$,
\begin{equation*}
\omega_{X}(\vartheta_i,\vartheta_j)=\omega_{X}(\tau_i,\tau_j)=0,\;\;\omega_{X}(\vartheta_i,\tau_j)=
\delta_{ij},\quad i,j=1,\dots,g.
\end{equation*}
This basis consists of differentials of the first kind $\vartheta_i$ and differentials of
the second kind $\tau_i$, uniquely characterized by the conditions
\begin{equation*}
\vartheta_{i}=\left(\delta_{ij}+O(z-z_{j})\right)\sd z\;\;\text{and}\;\;\tau_{i}=\left(\frac{\delta_{ij}}{(z-z_{j})^{2}}+O(z-z_{j})\right)\sd z,
\end{equation*}
where $z_{j}=z(P_{j})$ for a local coordinate $z$ at $P_{j}$ and
$i,j=1,\dots,g$.
\end{enumerate}
\end{theorem}

\begin{proof}
Let $(\theta)_{\infty}=\sum_{i=1}^{l}n_{i}Q_{i}$ be the polar divisor of
$\theta\in\Omega^{(\mathrm{2nd})}$, $n_{i}\geq 2$. Since $D$ is non-special, $h^{0}(K_{X}-D)=0$ and by Riemann-Roch formula we have $h^{0}(D+nQ_{i})=n+1$ for $n\geq 0$. Thus 
if $Q_{i}$ is not a point of $D$, there is $f_{i}\in\mathcal{L}_{D+(n_{i}-1)Q_{i}}$ such that
$$\mathrm{ord}_{Q_{i}}(\theta-df_{i})\geq 0.$$ 
If $Q_{i}$ is a point of $D$,  there is $f_{i}\in\mathcal{L}_{D+(n_{i}-1)Q_{i}} $ such that 
$$\mathrm{ord}_{Q_{i}}(\theta-df_{i})\geq -2.$$
(In this case, because $h^{0}(D)=1$, one can not adjust the principle part of $df_{i}$ at $Q_{i}$ to cancel posible second order pole of $\theta$).
Thus for $f=\sum_{i=1}^{l}f_{i}$ we have
$$(\theta-\sd f)\geq -2D,$$  which proves part (ii).

The dimension formula in part (i) easily follows from part (ii) since
\begin{gather*}
\dim_{\CC}\Omega^{(\mathrm{2nd})}\cap H^0(X, K_{X}+2D) \\=h^0(X, K_{X}+2D)-h^0(X, K_{X}+D)+h^0(X, K_{X}) \\=(3g-1)-(2g-1)+g=2g.
\end{gather*}

To prove part (iii), and the remaining statement in part (i), consider the linear map 
$$L: \Omega^{(\mathrm{2nd})}\cap H^0(X, K_{X}+2D)\to\CC^{2g},$$ defined as follows. For each
$\theta\in \Omega^{(\mathrm{2nd})}\cap H^0(X, K_{X}+2D)$, let $\alpha_{i}(\theta), \beta_{i}(\theta)\in\CC$ be such that near $P_{i}$
$$\frac{\theta}{\sd z}-\alpha_{i}(\theta)-\frac{\beta_{i}(\theta)}{(z-z_{i})^{2}}=O(z-z_{i}),$$
and put
$$L(\theta)=(\alpha_{1}(\theta), \beta_{1}(\theta),\dots,\alpha_{g}(\theta), \beta_{g}(\theta)).$$
Since $D$ is non-special, the map $L$ is injective and hence is an isomorphism, and we define $\vartheta_{i}$ and $\tau_{i}$ to have only non-zero components of $L$ to be, respectively, $\alpha_{i}=1$ and $\beta_{i}=1$.
\end{proof}
\begin{remark}
The choice of a non-special effective divisor $D$  on $X$ with $g$ distinct points $P_{i}$ and local coordinates is as an algebraic analogue of the choice of $a$-cycles on a Riemann surface. Correspondingly, differentials $\tau_{i}$ are analogues of differentials of the second kind with second-order poles, zero $a$-periods and normalized $b$-periods. The symplectic property of the basis $\{\vartheta_i,\tau_i\}_{i=1}^g$  is an analogue of the reciprocity laws for differentials of the first kind and the second kind (see\cite [Ch. 5, \S 1]{Iwasawa},  and \cite[Ch. VI, \S 3]{Kra}).
\end{remark}
\begin{remark} \label{dual-1}
Denote by $\Omega^{(\mathrm{2nd})}(2D)$ the  subspace in $\Omega^{(\mathrm{2nd})}$ spanned by $\tau_{i}$,
$$\Omega^{(\mathrm{2nd})}(2D)=\CC \tau_{1}\oplus\cdots\oplus\CC\tau_{g}.$$
Then $\Omega^{(\mathrm{2nd})}(2D)$ and $H^{0}(X,K_{X})$ are Lagrangian subspaces in $\Omega^{(\mathrm{2nd})}/\sd\mathcal{M}$, dual with respect to the pairing given by the symplectic form $\omega_{X}$.
\end{remark}

\section{Tangent space to Picard variety}\label{Picard}

We have a decomposition
\begin{equation}\label{dR-Ch}
H^{1}_{\mathrm{dR}}(X,\CC)=H^{1,0}(X,\CC)\oplus H^{0,1}(X,\CC)
\end{equation}
with the natural pairing 
$$H^{1,0}(X,\CC)\otimes H^{0,1}(X,\CC)\ni\alpha\otimes\beta\mapsto(\alpha,\beta)=\int_{X}\alpha\wedge\beta\in\CC.$$

The period map
$$H^{1,0}(X,\CC)\ni\vartheta\mapsto\int_{c}\vartheta\in\CC,\quad\text{where}\quad c\in H_{1}(X,\ZZ),$$
gives  canonical inclusion of the lattice $H_{1}(X,\ZZ)$ into  $H^{1,0}(X,\CC)^{\vee}$, the dual space to $H^{1,0}(X,\CC)$, and defines
the Albanese variety
$$\mathrm{Alb}(X)=H^{1,0}(X,\CC)^{\vee}/H_{1}(X,\ZZ).$$
Using the Dolbeault isomorphism and the exponential exact sequence of sheaves on $X$, we have for the Picard variety of line bundles of degree $0$
$$\mathrm{Pic}^{0}(X)=H^{0,1}(X,\CC)/H^{1}(X,\ZZ).$$

Thus holomorphic tangent space to $\mathrm{Pic}^{0}(X)$ can be identified with the vector space $H^{0,1}(X,\CC)$.

However, Theorem \ref{Chevalley} allows to describe the tangent space to Picard variety in purely algebro-geometric terms. Namely, we have the following simple result.

\begin{proposition}\label{iso-2D}
Each non-special effective divisor $D$ of degree $g$ defines an isomorphism
$$H^{0,1}(X,\CC)\simeq \Omega^{(\mathrm{2nd})}(2D).$$
\end{proposition}
\begin{proof}
It follows from Theorem \ref{Chevalley} part (iii), that the mapping
$$H^{0,1}(X,\CC)\ni\beta\mapsto\psi(\beta)=\sum_{i=1}^{g}(\vartheta_{i},\beta)\tau_{i}\in\Omega^{(\mathrm{2nd})}(2D)$$
satisfies
$$(\vartheta,\beta)=\omega_{X}(\vartheta,\psi(\beta))$$
for any $\vartheta\in H^{0}(X,K_{X})$, and is an isomorphism.
\end{proof}

Identifying $H^{1,0}(X,\CC)^{\vee}$ with $\Omega^{(\mathrm{2nd})}(2D)$, we get an inclusion of $H_{1}(X,\ZZ)$ into $ \Omega^{(\mathrm{2nd})}(2D)$, defined as follows.
Let $\theta_{c}$ be the $(0,1)$-component of the Poincar\'{e} dual of a cycle $c\in H_{1}(X,\ZZ)$, so
$$\int_{c}\vartheta=\int_{X}\vartheta\wedge\theta_{c}=(\vartheta,\theta_{c})\quad\text{for all}\quad\vartheta\in H^{1,0}(X,\CC).$$
Then
\begin{equation}\label{incl-2nd}
H_{1}(X,\ZZ)\ni c\mapsto\tau_{c}=\psi(\theta_{c})=\sum_{i=1}^{g}\int_{c}\vartheta_{i}\cdot \tau_{i}\in \Omega^{(\mathrm{2nd})}(2D),
\end{equation}
so
\begin{equation}\label{Alb-2nd}
\mathrm{Alb}(X)=\Omega^{(\mathrm{2nd})}(2D)/H_{1}(X,\ZZ).
\end{equation}
Thus a choice of a non-special effective divisor $D$ of degree $g$  allows to identify holomorphic tangent spaces to $\mathrm{Alb}(X)\simeq \mathrm{Pic}^{0}(X)\simeq\mathrm{Jac}(X)$ with the vector space $\Omega^{(\mathrm{2nd})}(2D)$ of the differentials of the second kind with poles in $D$.
Correspondingly, the holomorphic cotangent space is naturally identified with the vector space of $H^{1,0}(X,\CC)$ of differentials of the first kind, and the pairing with $\Omega^{(\mathrm{2nd})}(2D)$  is given by the symplectic form $\omega_{X}$. 
\section{The Baker-Akhiezer function} \label{BA}

Fix a non-special effective divisor $D_{0}=Q_{1} + \cdots + Q_{g}$ of degree $g$ and let $\{\vartheta_i\}_{i=1}^{g}$ be the basis of $H^{0}(X, K_{X})$ from Theorem \ref{Chevalley}, specialized to the divisor $D_{0}$. 
Consider the Abel-Jacobi map
$$X^{(g)}\ni D\to \mu^{(g)}(D)\in\mathrm{Jac}(X),$$
where $\mu^{(g)}$ is the Abel sum: for varying $D=P_{1} + \cdots + P_{g}$
\begin{equation}\label{abel-sum}
\mu^{(g)}(D)=\left(\sum_{i=1}^{g}\int_{Q_{i}}^{P_{i}}\vartheta_{1},\dots, \sum_{i=1}^{g}\int_{Q_{i}}^{P_{i}}\vartheta_{g}\right).
\end{equation}
 
 Choose local coordinates at $P_{i}$ and put $z_{i}=z(P_{i})$. It follows from \eqref{abel-sum} that $1$-forms $dz_{i}$ on $\mathrm{Jac}(X)$  at  the base point $\mu^{g}(D_{0})$ correspond to differentials  $\vartheta_{i}$, and the vector fields $\dfrac{\partial}{\partial z_{i}}$ to the differentials
of the second kind $\tau_{i}$ from Theorem \ref{Chevalley}. If divisor $D$ is also non-special, it follows from the group law on the Jacobian and Theorem \ref{Chevalley} that $\sd z_{i}$ and $\dfrac{\partial}{\partial z_{i}}$ at a point $\mu^{(g)}(D)$ are given by the symplectic basis of $\Omega^{(\mathrm{2nd})}\cap H^0(X, K_{X}+2D)$ from Theorem \ref{Chevalley}.

Equivalently, these vector fields on $\mathrm{Jac}(X)$ can be described using the formalism of Lax equations on algebraic curves, developed by the first author in \cite{K0, K1}. 
The main ingredients in  \cite{K0, K1} are stable vector bundles of rank $r$ and degree $rg$ and Lax operators, certain meromorphic $r\times r$ matrix-valued $1$-forms $L(z)dz$ on a Riemann surface and $r\times r$ meromorphic matrix-valued functions $M(z)$.  

Specialization to the Jacobian corresponds to the case $r=1$ and simplifies construction in \cite{K0, K1} dramatically. 
Namely, meromorphic $1$-forms $L(z)\sd z$  become differentials of the first kind $\vartheta\in H^{0}(X,K_{X})$, while  analogs of meromorphic functions $M(z)$  are defined 
as follows.

Consider the vector space 
$$\mathcal{L}_{D+D_{0}}=\{f\in\mathcal{M} : (f)+D+D_{0}\geq 0\}.$$ 
It follows from the Riemann-Roch theorem that $\dim_{\CC}\mathcal{L}_{D+D_{0}}=g+1$. Thus for any fixed choice of principal parts of $f$ at the points of $D_{0}$, not all of them are equal to zero, 
there is a unique, up to an inessential additive constant,  $f\in\mathcal{L}_{D+D_{0}}$ satisfying  
\begin{equation}\label{poles-P}
f(z)=\frac{\alpha_{i}}{z-z_{i}} +O(1),\quad z_{i}=z(P_{i})
\end{equation}
at all points of the divisor $D=P_{1}+\cdots+P_{g}$. Functions $f$, parametrized by their fixed principal parts at $D_{0}$, play the role of meromorphic functions $M(z)$ in case $r=1$; coefficients $\alpha_{i}$ depend on the principal parts at $D_{0}$.

We have a unique decomposition
$$\sd f=\tau-\tau_{0},$$
where  $\tau\in\Omega^{(\mathrm{2nd})}(2D)$ (see Remark \ref{dual-1}) and $(\tau_{0})+2D_{0}\geq 0$. By the residue theorem,
$$-\sum_{i=1}^{g}\mathrm{Res}_{P_{i}}(f\vartheta)=\omega_{X}(\vartheta,\tau)=\omega_{X}(\vartheta,\tau_{0}),\quad\vartheta\in H^{0}(X,K_{X}),$$
so the pairing (2.22) in \cite{K0}, given by the Krichever-Phong form, coincides with the pairing given by the symplectic form $\omega_{X}$.

A choice of a symplectic basis of $\Omega^{(\mathrm{2nd})}\cap H^0(X, K_{X}+2D)$ establishes a correspondence 
$$f\mapsto \cL_{f}= - \sum_{i=1}^{g}\alpha_{i}\dfrac{\partial}{\partial z_{i}}$$
between rational functions $f\in\mathcal{L}_{D+D_{0}}$ and vector fields on $\mathrm{Jac}(X)$. Along an integral curve $D(t)=P_{1}(t)+\cdots +P_{g}(t)$ of $\cL_{f}$, where $D(0)=D$, we have
\begin{equation}\label{dubrovin}
\dot{z}_{i}(t)=-\alpha_{i}(t),\quad i=1,\dots, g,
\end{equation}
where the dot stands for the $t$-derivative. 
In case when $X$ is a hyperelliptic curve, equations \eqref{dubrovin} are classical \emph{Dubrovin equations},  arising in the theory of finite-gap integration for the Korteweg-de Vries equation \cite{Dubrovin}, written in terms of the Abel transform. 
Using Dubrovin equations, we see that along the integral curve equations \eqref{poles-P} take the form
\begin{equation}\label{poles-P-t}
f_{t}(z)=-\frac{\dot{z}_{i}(t)}{z-z_{i}(t)} +O(1),\quad i=1,\dots,g.
\end{equation}
Thus integrating and introducing
$$\Psi(z)=\exp\left\{\int_{0}^{T}f_{t}(z)dt\right\},$$
we see from \eqref{poles-P-t} that $\Psi$ is a meromorphic function on $X\setminus D_{0}$ having simple poles only at $D$, simple zeros only at $D(T)$, and essential singularities at the points of $D_{0}$.
The function $\Psi$ is nothing but the celebrated \emph{Baker-Akhiezer function}, introduced by the first author in \cite{K}!

We leave it to the interested reader to describe by explicit formulas this connection between algebraic de Rham theorem and the integrable systems. 

\end{document}